\documentclass[11pt]{amsart}
\usepackage{mathpazo}
\usepackage{hyperref}
\usepackage{bm}

\usepackage{amssymb,amsmath,latexsym,amsthm,color}
\usepackage[cmtip,all]{xy}
\newtheorem{theorem}{Theorem}[section]
\newtheorem{lemma}[theorem]{Lemma}
\newtheorem{proposition}[theorem]{Proposition}

\newtheorem{predefinition}[theorem]{Definition}
\newenvironment{definition}{\begin{predefinition}\rm}{\end{predefinition}}
\newtheorem{preremark}[theorem]{Remark}
\newenvironment{remark}{\begin{preremark}\rm}{\end{preremark}}
\newtheorem{prenotation}[theorem]{Notation}

\newtheorem{preexample}[theorem]{Example}
\newenvironment{example}{\begin{preexample}\rm}{\end{preexample}}
\newtheorem{preclaim}[theorem]{Claim}

\newtheorem{prequestion}[theorem]{Question}
\newenvironment{question}{\begin{prequestion}\rm}{\end{prequestion}}
\newtheorem{preapplication}[theorem]{Application}
\newenvironment{application}{\begin{preapplication}\rm}{\end{preapplication}}

\numberwithin{equation}{section}

\DeclareFontEncoding{OT2}{}{}
  \newcommand{\textcyr}[1]{{\fontencoding{OT2}\fontfamily{wncyr}\fontseries{m}\fontshape{n}
     \selectfont #1}}
\newcommand{\Sha}{{\mbox{\textcyr{Sh}}}}

\newcommand \ZZ {{\mathbb Z}}

\newcommand \NN {{\mathbb N}}
\newcommand  \FF {{\mathbb F}}
\newcommand \GG {{\mathbb G}}
\newcommand \EE {{\mathbb E}}

\newcommand \dieu {{\mathbb D}}
\newcommand{\til}[1]{{\widetilde{#1}}}
\newcommand{\ang}[1]{\langle #1 \rangle}
\newcommand{\st}[1]{\left\{#1\right\}}

\newcommand{\rest}[1]{|_{#1}}

\def\aalpha{\boldsymbol{\alpha}}
\def\mmu{\boldsymbol{\mu}}

\global\let\ker\undefined
\DeclareMathOperator{\ker}{Ker}
\DeclareMathOperator{\frob}{Fr}
\DeclareMathOperator{\ver}{Ver}
\DeclareMathOperator{\im}{Im}
\DeclareMathOperator{\End}{End}

\DeclareMathOperator{\bt}{BT}
\newenvironment{alphabetize}{\begin{enumerate}

}{\end{enumerate}}

\begin{document}

\title[Superspecial rank]{Superspecial rank of supersingular abelian varieties and Jacobians}


\author{\sc Jeffrey D. Achter}
\address{Jeffrey D. Achter\\
Colorado State University\\
Fort Collins, CO, 80521\\
}
\email{achter@math.colostate.edu}
\urladdr{http://www.math.colostate.edu/~achter/}

\author{\sc Rachel Pries}
\address{Rachel Pries\\
Colorado State University\\
Fort Collins, CO, 80521\\
}
\email{pries@math.colostate.edu}
\urladdr{http://www.math.colostate.edu/~pries/}


\maketitle
\begin{abstract}
An abelian variety defined over an algebraically closed field $k$ of positive characteristic is supersingular if it is isogenous to a product of supersingular elliptic curves and is 
superspecial if it is isomorphic to a product of supersingular elliptic curves.
In this paper, the superspecial condition is generalized by defining the {\it superspecial rank} of an abelian variety, which is an invariant of its $p$-torsion.
The main results in this paper are about the superspecial rank of supersingular abelian varieties and Jacobians of curves.
For example, it turns out that the superspecial rank determines information about the decomposition of a supersingular abelian variety 
up to isomorphism; 
namely it is a bound for the maximal number of supersingular elliptic curves appearing in such a decomposition.
\end{abstract}

\maketitle

\section{Introduction}

If $A$ is a principally polarized abelian variety of dimension $g$ defined over an algebraically closed field $k$ of positive characteristic $p$, 
then the multiplication-by-$p$ morphism $[p]=\ver \circ \frob$ is inseparable.  
Typically, $A$ is {\it ordinary} in that the Verschiebung morphism $\ver$ is separable,
a condition equivalent to the number of $p$-torsion points of $A$ being $p^g$,
or the only slopes of the $p$-divisible group of $A$ being $0$ and $1$,
or the $p$-torsion group scheme of $A$ being isomorphic to $(\ZZ/p \oplus \mmu_p)^g$.

Yet the abelian varieties which capture great interest are those which are as far from being ordinary as possible. 
In dimension $g=1$, an elliptic curve is {\it supersingular} if it has no points of order $p$;
if the only slope of its $p$-divisible group is $1/2$; or if its $p$-torsion group scheme is isomorphic to
the unique local-local ${\rm BT}_1$ group scheme of rank $p^2$, which we denote by $I_{1,1}$. 

These characterizations are different for a principally polarized abelian variety $A$ of higher dimension $g$.
One says that $A$ has {\it $p$-rank $0$} when $A$ has no points of order $p$; 
that $A$ is {\it supersingular} when the only slope of its $p$-divisible group is $1/2$;
and that $A$ is {\it superspecial} when its $p$-torsion group scheme is isomorphic to $I_{1,1}^g$.
If $A$ is supersingular, then it has $p$-rank $0$, but the converse is false for $g \geq 3$.
If $A$ is superspecial, then it is supersingular, but the converse is false for $g \geq 2$.

The Newton polygon and Ekedahl-Oort type of an abelian variety usually do not determine the decomposition of the abelian variety.
In fact, for any prime $p$ and formal isogeny type $\eta$ other than the supersingular one, there exists an absolutely simple
abelian variety over $k$ having Newton polygon $\eta$ \cite{lenstraoort}.  On the other hand, consider the following results about
supersingular and superspecial abelian varieties.

\begin{theorem} (Oort)
Let $A/k$ be a principally polarized abelian variety.
\begin{enumerate}
\item Then $A$ is supersingular if and only if it is isogenous to a product of supersingular elliptic curves 
by \cite[Theorem 4.2]{O:sub} (which uses \cite[Theorem 2d]{tate:endo}).
\item Then $A$ is superspecial if and only if it is isomorphic to a product of supersingular elliptic curves  
\cite[Theorem 2]{oort75}, see also \cite[Theorem 4.1]{Nygaard}.
\end{enumerate}
\end{theorem}

The motivation for this paper was to find ways to measure the extent to which 
supersingular non-superspecial abelian varieties
decompose up to isomorphism.  
The $a$-number $a:={\rm dim}_k {\rm Hom}(\aalpha_p, A[p])$ gives some information about this; if $A$ has $p$-rank $0$, then
the number of factors in the decomposition of $A$ up to isomorphism is bounded above by the $a$-number, see \cite[Lemma 5.2]{SummerA}.
However, a supersingular abelian variety with large $a$-number could still be indecomposable up to isomorphism.

This paper is about another invariant of $A$, the {\it superspecial rank}, 
which we define in Section \ref{Sssdef} as the number of (polarized) factors of $I_{1,1}$ appearing in the $p$-torsion group scheme of $A$. 
In Proposition \ref{Pexists}, we determine which superspecial ranks occur for supersingular abelian varieties.
The superspecial rank of Jacobians also has an application involving Selmer groups, see Section \ref{Sselmer}.

In Section \ref{Sdecompose}, we define another invariant of $A$, the {\it elliptic rank}, 
which is the maximum number of elliptic curves appearing in a decomposition of $A$ up to isomorphism. 
In Proposition \ref{Pssrank=ssE}, we prove an observation of Oort which states that, for a supersingular abelian variety $A$, 
the elliptic rank equals the number of rank 2 factors in the $p$-divisible group $A[p^\infty]$. 
Proposition \ref{Psselliptic} states that the elliptic rank is bounded by the superspecial rank for an abelian
variety of $p$-rank $0$.
As a result, for an abelian variety $A$ of $p$-rank zero, the superspecial rank gives an upper bound for the maximal number of dimension one factors in a decomposition of $A$ up to isomorphism; this upper bound is most 
interesting for supersingular abelian varieties, which decompose completely up to isogeny.
 
In Section \ref{Sjacobian}, we apply this observation to prove some results about the superspecial rank and elliptic rank of Jacobians of curves.
For example, in characteristic $2$, Application \ref{App1general}
states that the superspecial rank of the Jacobian of any hyperelliptic curve of $2$-rank $r$ 
is bounded by $1 + r$, while its elliptic rank is bounded by $1+2r$.
The superspecial ranks of all the Hermitian curves are computed in Section \ref{Sherm}; 
in particular, when $n$ is even the elliptic rank of the Hermitian curve $X_{p^n}$ is zero.

The authors thank the organizers of the 2013 Journ\'ees Arithm\'etiques, the referee for valuable comments, 
Ritzenthaler for help with the French abstract, and 
Oort for sharing the idea for Proposition \ref{Pssrank=ssE} and more generally for being a source of inspiration for this work. The first-named author was partially supported by grants from the Simons Foundation (204164) and the NSA (H98230-14-1-0161 and H98230-15-1-0247).  The second-named author was partially supported by NSF grants DMS-11-01712 and DMS-15-02227.

\section{Notation}

All geometric objects in this paper are defined over an algebraically closed field $k$ of characteristic $p>0$.  
Some objects are defined over the ring $W(k)$ of Witt vectors over $k$.
Let $\sigma$ denote the Frobenius automorphism of $k$ and its lift to $W(k)$.
Let $A$ be a principally polarized abelian variety of dimension $g$ over $k$.
Here are some relevant facts about $p$-divisible groups and $p$-torsion group schemes.

\subsection{The $p$-divisible group}

By the Dieudonn\'e-Manin classification \cite{maninthesis}, there is an isogeny of $p$-divisible groups 
\[A[p^\infty] \sim \oplus_{\lambda=\frac{d}{c+d}} \til G_{c,d}^{m_\lambda},\] where $(c,d)$ ranges over pairs of relatively prime nonnegative integers, and $\til G_{c,d}$ 
denotes a $p$-divisible group of codimension $c$, dimension $d$,  and thus height $c+d$.
The Dieudonn\'e module $\til D_\lambda :=  \dieu_*(\til G_{c,d})$ (see \ref{subsecdefcartier} below) is a free $W(k)$-module of rank $c+d$.
Over $\operatorname{Frac}W(k)$, there is a basis $x_1, \ldots, x_{c+d}$ for $\til D_\lambda$ such that $F^{d}x_i=p^c x_i$.
The Newton polygon of $A$ is the data of the numbers $m_\lambda$; it admits an intepretation as the $p$-adic Newton polygon of the operator $F$ on $\dieu_*(A[p^\infty])$.

The abelian variety $A$ is {\it supersingular} if and only if $\lambda=\frac{1}{2}$ is the only slope of its $p$-divisible group $A[p^\infty]$.
Letting $\til{I}_{1,1}=\til G_{1,1}$ denote the $p$-divisible group of dimension $1$ and height $2$,
one sees that $A$ is supersingular if and only $A[p^\infty] \sim \til{I}_{1,1}^g$.

\subsection{The $p$-torsion group scheme}

The multiplication-by-$p$ morphism $[p]:A \to A$ is a finite flat morphism of degree $p^{2g}$.
The {\it $p$-torsion group scheme} of $A$ is 
\[A[p]=
\ker[p] = \ker(\ver\circ \frob),
\]
where $\frob:A \to A^{(p)}$ denotes the relative Frobenius morphism
and $\ver: A^{(p)} \to A$ is the Verschiebung morphism.  In fact, $A[p]$ is
a $\bt_1$ group scheme as defined in \cite[2.1, Definition 9.2]{O:strat}; it is killed by $[p]$, with
$\ker(\frob) = \im(\ver)$ and $\ker(\ver) = \im(\frob)$.

The principal polarization on $A$ induces a principal quasipolarization on $A[p]$,
i.e., an anti-symmetric isomorphism $\psi:A[p] \to A[p]^D$.  (This
definition must be modified slightly if $p=2$.)
Summarizing, $A[p]$ is a principally quasipolarized (pqp) $\bt_1$ group scheme of
rank $p^{2g}$.

Isomorphisms classes of pqp $\bt_1$ group schemes over $k$ (also known as
Ekedahl-Oort types) have been completely classified \cite[Theorem 9.4
\& 12.3]{O:strat}, building on
unpublished work of Kraft \cite{kraft} (which did not include polarizations) and of Moonen
\cite{M:group} (for $p \geq 3$). 
(When $p=2$, there are complications with the polarization which are resolved in \cite[9.2, 9.5, 12.2]{O:strat}.)

\subsection{Covariant Dieudonn\'e modules}
\label{subsecdefcartier}

The $p$-divisible group $A[p^\infty]$ and the $p$-torsion group scheme $A[p]$ can be described using covariant Dieudonn\'e theory; see e.g., \cite[15.3]{O:strat}.  Briefly, let $\til \EE = \til\EE(k) = W(k)[F,V]$ denote the non-commutative ring generated by semilinear operators $F$ and $V$ with relations
\begin{equation} \label{Efv}
FV=VF=p, \ F \lambda = \lambda^\sigma F, \ \lambda V=V \lambda^\sigma,
\end{equation} for all $\lambda \in W(k)$.  There is an equivalence of categories $\dieu_*$
between $p$-divisible groups over $k$ and $\til\EE$-modules which are free of finite rank over $W(k)$.

Similarly, let $\EE = \til \EE \otimes_{W(k)} k$ be the reduction of the Cartier ring mod $p$; it is a non-commutative ring $k[F,V]$ subject to the same constraints as \eqref{Efv}, except that $FV = VF = 0$ in $\EE$.  Again, there is an equivalence of categories $\dieu_*$ between finite commutative group schemes (of rank $2g$) annihilated by $p$ and $\EE$-modules of finite dimension ($2g$) over $k$.
If $M = \dieu_*(G)$ is the Dieudonn\'e module over $k$ of $G$, then  
a principal quasipolarization  $\psi:G \to G^D$ induces a 
a nondegenerate symplectic form 
\begin{equation}
\label{eqdefpolar}
\xymatrix{
\ang{\cdot,\cdot}:M \times M \ar[r]& k
}
\end{equation}
on the underlying $k$-vector space of $M$, subject to the additional constraint that, for all $x$ and $y$ in $M$,
\begin{equation}
\label{eqproppolar}
\ang{Fx,y} = \ang{x,Vy}^\sigma.
\end{equation}

If $A$ is the Jacobian of a curve $X$, then there is an isomorphism of $\EE$-modules between the {\em contravariant} 
Dieudonn\'e module over $k$ of ${\rm Jac}(X)[p]$ 
and the de Rham cohomology group $H^1_{\rm dR}(X)$ by \cite[Section 5]{Oda}.  The canonical principal polarization on $\operatorname{Jac}(X)$ then induces a canonical isomorphism $\dieu_*(\operatorname{Jac}(X)[p]) \simeq H^1_{\rm dR}(X)$; we will use this identification without further comment.

For elements $A_1, \ldots, A_r \in \EE$, 
let $\EE(A_1, \ldots, A_r)$ denote the left ideal $\sum_{i=1}^r \EE A_i$ of $\EE$ generated by $\{A_i \mid 1 \leq i \leq r\}$.

\subsection{The $p$-rank and $a$-number} \label{Sprankanumber} \label{Sanumber}

For a $\bt_1$ group scheme $G/k$, 
the {\it $p$-rank} of $G$ is $f={\rm dim}_{\FF_p} {\rm Hom}(\mmu_p, G)$
where $\mmu_p$ is the kernel of Frobenius on $\GG_m$.
Then $p^f$ is the cardinality of $G(k)$.
The {\it $a$-number} of $G$ is 
\[a={\rm dim}_k {\rm Hom}(\aalpha_p, G),\]
where $\aalpha_p$ is the kernel of Frobenius on $\GG_a$.
It is well-known that $0 \leq f \leq g$ and $1 \leq a +f \leq g$.

Moreover, since $\mmu_p$ and $\aalpha_p$ are both simple group schemes,
the $p$-rank and $a$-number are additive;
\begin{equation}
\label{eqfadditive}
f(G\oplus H) = f(G)+f(H)\text{ and }a(G\oplus H) = a(G)+a(H).
\end{equation}

If $\til G$ is a $p$-divisible group, its $p$-rank and $a$-number are those of its $p$-torsion; $f(\til G) = f(\til G[p])$ and $a(\til G) = a(\til G[p])$.  Similarly, if $A$ is an abelian variety, then $f(A) = f(A[p])$ and $a(A) = a(A[p])$.

\subsection{The Ekedahl-Oort type} \label{Seotype}

As in \cite[Sections 5 \& 9]{O:strat}, the isomorphism type of a pqp ${\rm BT}_1$ group scheme 
$G$ over $k$ can be encapsulated into combinatorial data.
If $G$ is symmetric with rank $p^{2g}$, then there is a {\it final filtration} $N_1 \subset N_2 \subset \cdots \subset N_{2g}$ 
of ${\mathbb D}_*(G)$ as a $k$-vector space which is stable under the action of $V$ and $F^{-1}$ such that $i={\rm dim}(N_i)$ \cite[5.4]{O:strat}.

The {\it Ekedahl-Oort type} of $G$ is 
\[\nu=[\nu_1, \ldots, \nu_g], \ {\rm where} \ {\nu_i}={\rm dim}(V(N_i)).\]
The $p$-rank is ${\rm max}\{i \mid \nu_i=i\}$ and the $a$-number equals $g-\nu_g$.
There is a restriction $\nu_i \leq \nu_{i+1} \leq \nu_i +1$ on the Ekedahl-Oort type.
There are $2^g$ Ekedahl-Oort types of length $g$ since all sequences satisfying this restriction occur.   
By \cite[9.4, 12.3]{O:strat}, there are bijections between (i) Ekedahl-Oort types of length $g$; (ii) pqp ${\rm BT}_1$ group schemes over $k$ of rank $p^{2g}$;
and (iii) pqp Dieudonn\'e modules of dimension $2g$ over $k$.

\begin{example}\label{exi11} {\em The group scheme $I_{1,1}$.}
There is a unique ${\rm BT}_1$ group scheme of rank $p^2$ which has $p$-rank $0$, which we denote $I_{1,1}$.
It fits in a non-split exact sequence 
\begin{equation}
\label{eqdefi11}
0 \to \aalpha_p \to I_{1,1} \to \aalpha_p \to 0.
\end{equation}
The structure of $I_{1,1}$ is uniquely determined over $\overline{\FF}_p$ by 
this exact sequence.  The image of $\aalpha_p$ is the kernel of $\frob$ and $\ver$.
The Dieudonn\'e module of $I_{1,1}$ is $$M_{1,1} := \dieu_*(I_{1,1}) \simeq \EE/\EE(F+V).$$
If $E$ is a supersingular elliptic curve, then the $p$-torsion group scheme $E[p]$ is isomorphic to $I_{1,1}$.
\end{example}

\section{Superspecial rank} \label{Sssrank}

Let $A$ be a principally polarized abelian variety defined over an algebraically closed field $k$ of characteristic $p >0$.

\subsection{Superspecial}  

First, recall the definition of the superspecial property.

\begin{definition}
One says that $A/k$ is {\it superspecial} if it satisfies the following equivalent conditions:
\begin{enumerate}
\item The $a$-number of $A$ equals $g$.
\item The group scheme $A[p]$ is isomorphic to $I_{1,1}^g$.
\item The Dieudonn\'e module over $k$ of $A[p]$ is isomorphic to $M_{1,1}^g$.
\item $A$ is isomorphic (as an abelian variety without polarization) to the product of $g$ supersingular elliptic curves.
\end{enumerate}
\end{definition}
 
A superspecial abelian variety is defined over $\overline{\FF}_p$, and thus over a finite field.
For every $g \in \NN$ and prime $p$, the number of superspecial principally polarized abelian varieties of dimension $g$ defined over 
$\overline{\FF}_p$ is finite and non-zero.

\subsection{Definition of superspecial rank} \label{Sssdef}

Recall (Example \ref{exi11}) that the $p$-torsion group scheme of a supersingular elliptic curve is isomorphic to $I_{1,1}$, 
the unique local-local pqp ${\rm BT}_1$ group scheme of rank $p^2$.  
From \eqref{eqdefi11}, it follows that $I_{1,1}$ is not simple as a group
scheme.  However, $I_{1,1}$ is
simple in the category of $\bt_1$ group schemes since $\aalpha_p$ is not a $\bt_1$ group scheme. 

\begin{definition}
Let $G/k$ be a $\bt_1$ group scheme.  
A {\it superspecial factor} of $G$ is a group scheme $H \subset G$ with $H \simeq I_{1,1}^s$.
\end{definition}

By the equivalence of categories $\dieu_*$, 
superspecial factors of $G$ of rank $2s$ are in bijection with
$\EE$-submodules $N \subset \dieu_*(G)$ with $N \simeq (\EE/\EE(F+V))^s$; 
we call such an $N$ a {\it superspecial factor} of $M=\dieu_*(G)$.

Now suppose $(G, \psi)/k$ is a pqp $\bt_1$ group scheme.
 A superspecial factor $H$ of $G$ is
{\it polarized} if the isomorphism $\psi: G \to G^D$ restricts to an
isomorphism $\psi_H: H \to G^D \twoheadrightarrow H^D$.  
Equivalently, a superspecial
factor $N$ of $(\dieu_*(G), \ang{\cdot,\cdot})$ is polarized if
the nondegenerate symplectic form $\ang{\cdot,\cdot}:M \times M \to k$
restricts to a non-degenerate symplectic form $\ang{\cdot,\cdot}:N
\times N \to k$.

\begin{definition}
Let $G=(G,\psi)/k$ be a pqp $\bt_1$ group scheme.
The {\it superspecial rank} $s(G)$ of $G$ is the largest integer $s$
for which $G$ has a polarized superspecial factor of rank $2s$.
\end{definition}

Since $I_{1,1}$ is simple in the category of $\bt_1$ group schemes, the superspecial rank $s$ 
has an additive property similar to that for the $p$-rank and $a$-number \eqref{eqfadditive};
if $G$ and $H$ are pqp $\bt_1$ group schemes, then 
\begin{equation}
\label{eqsadditive}
s(G\oplus H) = s(G)+s(H).
\end{equation}

A $\bt_1$ group scheme $G$ may fail to be simple (i.e., admit a nontrivial
$\bt_1$ subgroup scheme $0\subsetneq H \subsetneq G$) and yet still be
indecomposable (i.e., admit no isomorphism $G \simeq H\oplus K$ with
$H$ and $K$ nonzero).  This distinction vanishes in the
category of pqp $\bt_1$ group schemes:

\begin{lemma}
\label{lemdecompbt1}
Let $G/k$ be a pqp $\bt_1$ group scheme, and let $H\subset G$ be a pqp
$\bt_1$ sub-group scheme.  Let $N = \dieu_*(H) \subseteq M = \dieu_*(G)$, and let $P$
be the orthogonal complement of $N$ in $M$.  
Then $P$ is a pqp
sub-Dieudonn\'e module of $M$, and $G$ admits a decomposition $G\simeq
H\oplus K$ as pqp $\bt_1$ group schemes, where $K\subseteq G$ is the sub-group scheme with  $\dieu_*(K) = P$.
\end{lemma}

Lemma \ref{lemdecompbt1} is essentially present in \cite[Section 5]{kraft}; see, e.g., \cite[9.8]{O:strat}.

\begin{proof}
The $k$-vector space $P$ is an $\EE$-module if it is stable under $F$ and $V$. 
It suffices to check that, for $\beta\in P$,  $F \beta \in P$ and $V \beta \in P$.
If $\alpha \in N$, the relation \eqref{eqproppolar} implies that
\[
\ang{F\beta, \alpha} = \ang{\beta, V \alpha}^\sigma = 0^\sigma = 0
\]
and
\[
\ang{V\beta, \alpha} = \ang{\beta,F \alpha}^{\sigma^{-1}} = 0^{\sigma^{-1}}=0.
\]
Thus $F \beta$ and $V \beta$ are in the orthogonal complement $P$ of
$N$.

Since $H$ is polarized, the restriction of $\ang{\cdot,\cdot}$ to $N$ is perfect and so the 
restriction of $\ang{\cdot,\cdot}$ to $P$ is perfect as well.  Since
$\dieu_*$ is an equivalence of categories, there is a decomposition
$G\simeq H \oplus K$ as pqp group schemes.  It remains to verify that
$K$ is a $\bt_1$ group scheme, i.e., that $\ker(\frob) = \im(\ver)$ and $\ker(\ver)
= \im(\frob)$.  In terms of Dieudonn\'e modules, this is equivalent to the property that
$\ker F\rest P = V(P)$ and $\ker V\rest P = F(P)$.  This, in turn,
follows from the analogous statement for $M$ and $N$ and from the fact
that the decomposition $M = N \oplus P$ is stable under $F$ and $V$.
\end{proof}

\begin{lemma}
\label{lemsplitss}
Let $G/k$ be a pqp $\bt_1$ group scheme of $p$-rank $f$ and $a$-number $a$, and let $H\subset G$ be a maximal 
polarized superspecial factor.  Then $G \simeq H \oplus K$ for a pqp
$\bt_1$ group scheme $K$ with respective $p$-rank, superspecial rank and $a$-number $f(K)=f$, $s(K) =0$, and $a(K) = a-s$.
\end{lemma}

\begin{proof}
The existence of the decomposition $G\simeq H\oplus K$ follows from Lemma
\ref{lemdecompbt1}; the assertions about the $p$-rank, superspecial
rank and $a$-number of $K$ follow from the additivity of these quantities, \eqref{eqfadditive} and
\eqref{eqsadditive}.
\end{proof}

Since one can always canonically pull off the \'etale and toric components of a finite group scheme over a perfect field, Lemma \ref{lemsplitss} admits a further refinement:

\begin{lemma}
\label{Lpulloffs}
Let $G/k$ be a pqp $\bt_1$ group scheme with $f(G)=f$, $s(G) = s$, and $a(G) = a$.  Then there is a local-local pqp $\bt_1$ 
group scheme $B$ such that
\[
G \simeq (\ZZ/p\oplus \mmu_p)^f \oplus I_{1,1}^s \oplus B
\]
where $f(B) = s(B) = 0$ and $a(B) = a-s$.
\end{lemma}

\begin{proof}
Since $k$ is perfect and $G$ is self-dual, there is a canonical decomposition of pqp group schemes $G \simeq (\ZZ/p\oplus \mmu_p)^f \oplus H$.  Then $f(H) = 0$, $s(H) = s(G)$, and $a(H) = a(G)$.  Now invoke Lemma \ref{lemsplitss}.
\end{proof}

%
%
%
%
Let $A$ be a principally polarized abelian variety of dimension $g$.  On one hand, $A$ is superspecial if and only if $s(A[p]) = g$.  On the other hand, if $A$ is ordinary, then $s(A[p]) =0$.  More generally:

\begin{lemma}
\label{lemsanda}
Let $G/k$ be a pqp $\bt_1$ group scheme of rank $p^{2g}$; let $f = f(G)$, $a=a(G)$, and $f = f(G)$.
\begin{alphabetize}
\item Then $0 \le s \le a \le g-f$.
\item If $a = g-f$, then $G \simeq (\ZZ/p\oplus \mmu_p)^f \oplus I_{1,1}^a$ and $s=a$.
\item If $a\not = g-f$, then $s<a$.
\end{alphabetize}
\end{lemma}

\begin{proof}
Write $G \simeq (\ZZ/p \oplus \mmu_p)^f \oplus B_1$ with $B_1 \simeq I_{1,1}^s \oplus B$ as in Lemma \ref{Lpulloffs}.
\begin{alphabetize}
\item Then $a \leq g-f$, since (using additivity) $a(G) = a(B_1)$, and $B_1$ has rank $p^{2(g-f)}$.  Moreover,  $s \le a$ 
since $a(I_{1,1}^s) = s$.
\item This is true since the only
pqp ${\rm BT}_1$ group scheme of rank $p^{2(g-f)}$ with $p$-rank $0$ and $a$-number $g-f$ is $I_{1,1}^{g-f}$, 
which has superspecial rank $g-f$ by definition.
\item The hypothesis $a \not = g-f$ implies that $B$ is non-trivial. 
Then $a > s$ since the $a$-number of the local-local group scheme $B$ is at least $1$.
\end{alphabetize}
\end{proof}

\subsection{Unpolarized superspecial rank}

If $G/k$ is a $\bt_1$ group scheme, or indeed any $p$-torsion finite
commutative group scheme, then there is also an obvious notion of an
{\em unpolarized} superspecial rank, namely, the largest $u$ such that
there is an inclusion $I_{1,1}^{u}\hookrightarrow G$. In this section,
we briefly explore some of the limitations of this notion.

For integers $r,s \ge 1$, let $J_{r,s}$ be the $\bt_1$ group scheme with
Dieudonn\'e module
\[
M_{r,s} := \dieu_*(J_{r,s}) = \EE/\EE(F^r+V^s).
\]

\begin{lemma}
\label{lemirs}
Suppose $r,s \ge 2$.  Then
\begin{alphabetize}
 \item $J_{r,s}$ is an indecomposable local-local $\bt_1$ group scheme.
\item There exists an inclusion $\iota:I_{1,1} \hookrightarrow
  J_{r,s}$.
\end{alphabetize}
\end{lemma}

\begin{proof}
Part (a) is standard.  Indeed, using the relations $F^r = -V^s$ and
$FV = VF = p$, one sees that $F$ and $V$ act nilpotently, and thus
$J_{r,s}$ is local-local; in particular, it has $p$-rank zero.  
Note that $M_{r,s}$ is generated over $\EE$ by
a single element $x$ such that $F^r x = -V^sx$. 
It follows that $a(J_{r,s}) = 1$.  
The additivity relation \eqref{eqfadditive} now implies that $M_{r,s}$ is
indecomposable.

For (b), let $y \in M_{r,s}$ be an element such that $Fy = - Vy \not =
0$.  Since $r,s \geq 2$, the element $y= F^{r-1}x+V^{s-1}x$ is suitable.  Then
there is an inclusion $\iota_*:M_{1,1} \to M_{r,s}$ which sends a
generator of $M_{1,1}$ to $y$.
\end{proof}

If $r = s$, then $J_{r,s}$ is self-dual and admits a principal
quasipolarization; in this case, let $H_{r,s} = J_{r,s}$.  If $r\not = s$,
then the Cartier dual of $J_{r,s}$ is
$J_{s,r}$; in this case, $H_{r,s}:= J_{r,s} \oplus J_{s,r}$ admits a principal
quasipolarization. 

In spite of Lemma \ref{lemirs}, we find:
\begin{lemma}
\label{lemsshrs0}
Suppose $r, s \ge 2$.  For any principal quasipolarization on $H_{r,s}$, the superspecial rank of $H_{r,s}$ is zero.
\end{lemma}

\begin{proof}
If $r=s$, this is immediate, since $H_{r,r}$ is indecomposable by Lemma
\ref{lemirs} and yet a polarized superspecial factor of positive
rank would induce a factorization by Lemma \ref{lemdecompbt1}.

Now suppose $r\not = s$. 
The argument used in the classification of polarizations on superspecial
$p$-divisible groups in \cite[Section 6.1]{LO} shows that for some $u \in
\st{1,2}$, there exists an inclusion
$\iota:I_{1,1}^u \hookrightarrow H_{r,s}$ with $G := \iota(I_{1,1}^u)$
polarized.  If $G$ is contained in either $J_{r,s}$ or $J_{s,r}$ (and
in particular if $u=1$), then we may argue as before.  Otherwise,
consider the sum of $G$ and $J_{r,s}$ inside $H_{r,s}$, which is {\em
  not} direct since $G\cap J_{r,s} \simeq I_{1,1}$ is nonempty.  By
Lemma \ref{lemdecompbt1}, $G$ has a complement $K$ in $G+J_{r,s}$.
Then $J_{r,s} \simeq I_{1,1}\oplus K$, contradicting the
indecomposability of $J_{r,s}$.
\end{proof}

\subsection{Superspecial ranks of abelian varieties}
\label{SssA}

If $A/k$ is a principally polarized abelian variety, we define its
superspecial rank to be that of its $p$-torsion group scheme;
$s(A) = s(A[p])$.  Lemma \ref{lemsanda} gives constraints between the
$p$-rank, $a$-number, and superspecial rank of $A$.  It turns out that
these are the only constraints on $f$, $a$ and $s$:

\begin{proposition}
Given integers $g,f,a,s$ such that $0 \leq s < a < g-f$, 
there exists a principally polarized abelian variety $A/k$ of dimension $g$ with $p$-rank $f$, $a$-number $a$ and superspecial rank $s$.
\end{proposition}

\begin{proof}
By \cite[Theorem 1.2]{O:strat}, 
it suffices to show that there exists a pqp ${\rm BT}_1$ group scheme $G$ of rank $p^{2g}$ 
with $p$-rank $f$, $a$-number $a$ and superspecial rank $s$.
Set 
\[g_1=g-f-s, \ {\rm and} \ a_1=a-s,\] 
and note that $a_1 \geq 1$ and $g_1-a_1 \geq 1$ by hypothesis.
Considering
\[G = (\ZZ/p \oplus \mmu_p)^f \oplus I_{1,1}^s \oplus B,\]
together with the product polarization, 
allows one to reduce to the case of finding 
a pqp ${\rm BT}_1$ group scheme $B$ of rank $p^{2g_1}$ with $p$-rank $0$, $a$-number $a_1$ and superspecial rank $0$.
This is possible as follows.

Consider the word $w$ in $F$ and $V$ given by 
\[w=F^{g_1-a_1+1}(VF)^{a_1-1}V^{g_1-a_1+1}.\]
Then $w$ is simple and symmetric with length $2g_1$, and thus the corresponding $\bt_1$ group scheme admits a canonical principal quasipolarization \cite[9.11]{O:strat}.
Let $L_1, \ldots, L_{2g_1} \in \{F, V\}$ be such that $w=L_1 \cdots L_{2g_1}$. 
Consider variables $z_1, \ldots, z_{2g_1}$ with $z_{2g_1+1}=z_1$.
As in \cite[Section 9.8]{O:strat}, the word $w$ defines the structure of a Dieudonn\'e module on $N_w=\oplus_{i} k \cdot z_i$ as follows:
if $L_i=F$, let $F(z_i)=z_{i+1}$ and $V(z_{i+1})=0$; if $L_i=V$, let $V(z_{i+1})=z_i$ and $F(z_i)=0$.

The $a$-number is the number of generators for $N_w$ as an $\EE$-module.
By construction, $N_w$ has $a$-number $a_1$.  Since $g_1-a_1+1 \geq 2$, then $N_w$ has superspecial rank $0$.
\end{proof}


We now focus on supersingular abelian varieties

\begin{lemma} \label{sswiths=0}
For every $g \geq 2$ and prime $p$, a generic supersingular principally polarized abelian variety 
of dimension $g$ over $k$ has superspecial rank $0$.
\end{lemma}

\begin{proof}
A generic supersingular principally polarized abelian variety has $p$-rank $0$ and $a$-number $1$ \cite[Section 4.9]{LO}. 
This forces its Ekedahl-Oort type to be $[0,1, \ldots, g-1]$,  its Dieudonn\'e module to be 
$M_{g,g}$, and its superspecial rank to be zero (Lemma \ref{lemsshrs0}) since $g \ge 2$.
\end{proof}
  
It is not difficult to classify the values of the supersingular rank which occur for supersingular abelian varieties.

\begin{proposition} \label{Pexists}
For every $g \geq 2$ and prime $p$, there exists a supersingular principally polarized abelian variety of dimension $g$ over $k$ 
with superspecial rank $s$ if and only if $0 \leq s \leq g-2$ or $s=g$.
\end{proposition}

\begin{proof}
It is impossible for the superspecial rank to be $g-1$ since there are no local-local 
pqp ${\rm BT}_1$ group schemes of rank $p^2$ other than $I_{1,1}$. 

For the reverse implication, recall that there exists a supersingular principally polarized abelian variety $A_1/k$ of dimension $g-s$ with $a=1$. 
Its Dieudonn\'e module is $M_{g-s,g-s}$.  In particular, $s(A_1)=0$ as long as $s \le g-2$ (Lemma \ref{lemsshrs0}).
Let $E$ be a supersingular elliptic curve.
Then $A=E^s \times A_1$, together with the product polarization, is a supersingular principally polarized 
abelian variety over $k$ with dimension $g$ and $s(A)=s$.
\end{proof}

\begin{example}
Let $A/k$ be a supersingular principally polarized abelian variety of dimension $3$.  
Then the $a$-number $a=a(A)$ satisfies $1 \leq a \leq 3$.
\begin{alphabetize}
\item{If $a=1$,} then $A[p] \simeq J_{3,3}$,  which has superspecial rank $s=0$.
\item{If $a=2$,} then $A[p^\infty] \simeq \til G_{1,1} \times \til Z$ where $\til Z$ is supergeneral of height $4$ and $a(\til Z)=1$ \cite{odaoort}.  Then $s(\til Z[p]) = 0$ (Lemma \ref{lemsanda}(c)) and thus $s(A)=1$.
\item{If $a=3$,} then $A$ has superspecial rank $s=3$.
\end{alphabetize}
\end{example}

\subsection{Application of superspecial rank to Selmer groups} \label{Sselmer}

Here is another motivation for studying the superspecial rank of Jacobians.
The superspecial rank equals the rank of the Selmer group associated with a particular isogeny 
of function fields in positive characteristic.
Let $K$ be the function field of a smooth projective connected curve $X$ over $k$.
Let ${\mathcal E}$ be a constant supersingular elliptic curve over $K$.
Consider the multiplication-by-$p$ isogeny $f=[p]: {\mathcal E} \to {\mathcal E}$ of abelian varieties over $K$.

Recall the Tate-Shafarevich group
\[\Sha(K, {\mathcal E})_f={\rm Ker}(\Sha(K, {\mathcal E}) \stackrel{f}{\to} \Sha(K, {\mathcal E})),\] where 
\[\Sha(K,{\mathcal E})={\rm Ker}(H^1(K, {\mathcal E}) \to \prod_{v} H^1(K_v, {\mathcal E}))\]
and $v$ runs over all places of $K$.
The Selmer group ${\rm Sel}(K, f)$ is the subset of elements of $H^1(K, {\rm Ker}(f))$ whose restriction is in the
image of \[{\rm Sel}(K_v, f) = {\rm Im}({\mathcal E}(K_v) \to H^1(K_v, {\rm Ker}(f))),\] for all $v$.
There is an exact sequence 
\[0 \to {\mathcal E}(K)/f({\mathcal E}(K)) \to {\rm Sel}(K,f) \to \Sha(K, {\mathcal E})_f \to 0.\]

Here is an earlier result, rephrased using the terminology of this paper, which provides motivation for studying the superspecial rank.

\begin{theorem} (Ulmer)
The rank of ${\rm Sel}(K, [p])$ is the superspecial rank of ${\rm Jac}(X)$ \cite[Proposition 4.3]{Ulmer}.
\end{theorem}

\section{Elliptic curve summands of abelian varieties} \label{Sdecompose}

Let $A/k$ be a principally polarized abelian variety of dimension $g$.
In this section, we define the elliptic rank of $A$ to be the maximum number of elliptic curves appearing in a decomposition of $A$ up to isomorphism.  
When $A$ has $p$-rank $0$, the 
elliptic rank is bounded by the superspecial rank, Proposition \ref{Psselliptic}.
Proposition \ref{Pssrank=ssE} states that the elliptic rank is the number of rank $2$ factors
in the $p$-divisible group $A[p^\infty]$ when $A$ is supersingular. 

\subsection{Elliptic rank}

\begin{definition}
The {\it elliptic rank} $e(A)$ of $A$ is
\begin{equation}
\label{eqdefe}
e(A):={\rm max} \{e \mid \iota: A \stackrel{\simeq}{\to}  A_1 \times (\times_{i=1}^e E_i)\},
\end{equation}
where $E_1, \ldots, E_e$ are elliptic curves, $A_1$ is an abelian variety of dimension $g-e$, 
and $\iota$ is an isomorphism of abelian varieties over $k$. 
\end{definition}
(We remind the reader that many ``cancellation problems'' for abelian varieties have negative answers \cite{shioda77}, and that the abelian variety $A_1$ in \eqref{eqdefe} is not necessarily unique.)

Here are some properties of the elliptic rank.

\begin{proposition} \label{Psselliptic}
If $A$ has $p$-rank $0$, then the elliptic rank is bounded by the superspecial rank: $e(A) \leq s(A)$.
\end{proposition}
  
\begin{proof}
If $A$ has $p$-rank $0$, then the elliptic curves $E_1, \ldots, E_e$ in a maximal decomposition of $A$ are supersingular.
Each supersingular curve in the decomposition contributes a factor of $\EE/\EE(F+V)$ to the Dieudonn\'e module $\dieu_*(A[p])$.
\end{proof}

The proof of Proposition \ref{Pexists} shows that, 
for every $g \geq 2$ and prime $p$, there exists a supersingular principally polarized abelian variety of dimension $g$ over $k$ 
with elliptic rank $e$ if and only if $0 \leq e \leq g-2$ or $e=g$.

\begin{remark} \label{Rabssimple}
It is clear that $e(A) =0$ if $A$ is simple and ${\rm dim}(A)>1$.  
Recall from \cite{lenstraoort} that there exists a simple
abelian variety $A$ with formal isogeny type $\eta$, for each non-supersingular Newton polygon $\eta$.
It follows from Proposition \ref{Psselliptic} that there exist abelian varieties $A$ with $s(A) > 0$ and $e(A) =0$ for all dimensions $g \geq 4$.
\end{remark}

\subsection{Superspecial rank for $p$-divisible groups}

We briefly sketch a parallel version of superspecial rank in the category of
$p$-divisible groups, rather than $p$-torsion group schemes. Many of
the notions and results in Section \ref{Sssdef} generalize to
truncated Barsotti-Tate groups of arbitrary level, and indeed to
Barsotti-Tate, or $p$-divisible, groups.

Let $\til G$ be a pqp $p$-divisible group, and
let $\til H \subseteq \til G$ be a sub-$p$-divisible group.  We say that
$\til H$ is polarized if the principal quasipolarization on $\til G$
restricts to one on $\til H$.  Lemma \ref{lemdecompbt1} admits an
analogue for $p$-divisible groups; for such an $\til H$, there exists a
pqp complement $\til K$ such that $\til G \simeq
\til H \oplus \til K$.

Let
$\til I_{1,1}$ be the $p$-divisible group whose Dieudonn\'e module is
\[
\til M_{1,1} = \dieu_*(\til I_{1,1}) \simeq \til\EE/\til\EE(F+V);
\]
then $\til I_{1,1}[p] \simeq I_{1,1}$. 

With this preparation, we define the superspecial rank $\til s(\til
G)$ of a pqp $p$-divisible group $\til G$ as the largest value of $s$
for which there exists a sub-pqp $p$-divisible group of $\til G$
isomorphic to $\til I_{1,1}^{s}$.

Since a decomposition of a $p$-divisible group induces a decomposition
on its finite levels, it follows that
\begin{equation}
\til s(\til G) \le s(\til G[p]).
\end{equation}
Similarly, if $A/k$ is a principally polarized abelian variety, then any
decomposition of $A$ induces a decomposition of its $p$-divisible
group.  So if $A$ has $p$-rank $0$, then
\begin{equation}
e(A) \le \til s(A[p^\infty]).
\end{equation}

We thank Oort for suggesting the following result:

\begin{proposition}\label{Pssrank=ssE}
Let $A/k$ be a supersingular principally polarized abelian variety.
Then
\[
e(A) = \til s(A[p^\infty]).
\]
\end{proposition}

\begin{proof} 
Let $\til M$ be
the Dieudonn\'e module $\til M= \dieu_*(A[p^\infty])$, and let $E/k$ be a supersingular elliptic curve.  Since $A$ is
principally polarized, $\til M$ is principally quasipolarized.
Let $\til s = \til s(A[p^\infty])$. 
By the same proof as for Lemma \ref{lemdecompbt1},
there is a decomposition of pqp Dieudonn\'e modules
\begin{equation}
\label{eqdecomptilM}
\til M \simeq \til M_{1,1}^{\til s} \oplus \til N,
\end{equation}
where $\til N$ has superspecial rank zero.  By \cite[Theorem 6.2]{ogusSS}, since $\til M$
is supersingular, \eqref{eqdecomptilM} induces a corresponding
decomposition 
\begin{equation}
\label{eqdecompabvar}
A \simeq E^{\til s} \oplus A_1.
\end{equation}
where $A_1$ is a principally polarized abelian variety of dimension $g-\til{s}$ with $\til s(A_1[p^\infty]) = 0$.
Thus $e(A) \geq \til{s}$ and the result follows.
\end{proof}

\begin{remark}
In fact, it is not hard to give a direct proof that the existence of decomposition \eqref{eqdecomptilM} implies the existence of \eqref{eqdecompabvar}. Indeed, 
since $A$ is supersingular, there exists an isogeny $\psi:E^g \to A$, which
induces an isogeny of $p$-divisible groups $\psi[p^\infty]: \til
I_{1,1}^g \to A[p^\infty]$.   Let $H = \ker (\psi[p^\infty])$; it is a
finite group scheme, and is thus also a sub-group scheme of $E^g$.  Since $\End(\til I_{1,1})$ is a maximal order in
a division ring over $\ZZ_p$, it is a (noncommutative) principal
ideal domain (see also \cite[p.\ 335]{Li:ss}).  By the theory of
elementary divisors for such rings (e.g., \cite[Chapter 3, Theorem
18]{jacobsonrings}), there is an isomorphism $\til I_{1,1}^g \simeq
\til I_{1,1}^{\til s} \times \til I_{1,1}^{g-\til s}$ under which $H$
is contained in $0\times \til I_{1,1}^{g-\til s}$.  Since
$\End(E^g[p^n]) \simeq \End(\til I_{1,1}^g [p^n])$ for each $n \in \NN$, there is an
analogous decomposition $E^g \simeq E^{\til s} \times E^{g-\til s}$
under which $H$ is contained in $0\times E^{g-\til s}$.  Then $A =
E^g/N \simeq E^{\til s} \oplus A_1$, where $A_1$ is supersingular but has
superspecial rank zero.
\end{remark}

\subsection{An open question}

Consider a principally polarized abelian variety $A/k$.
By Remark \ref{Rabssimple}, if $A$ is not supersingular, then it can 
be absolutely simple ($e(A)=0$) and yet have positive superspecial rank ($s(A)>0$).  (Similarly, if $A$ admits ordinary elliptic curves as factors, then it is possible to have $e(A)>0$ while $s(A)=0$.)

However, if $A$ has $p$-rank $0$, there are {\em a priori} inequalities
\[
e(A) \le \til s(A[p^\infty]) \le s(A[p]).
\]
Proposition \ref{Pssrank=ssE} shows the first inequality is actually an equality when $A$ is supersingular.  
This leads one to ask the following:

\begin{question} \label{Qask}
\begin{enumerate}
\item If $A/k$ is supersingular, is $e(A)=s(A)$?
\item If $\til G$ is a supersingular pqp $p$-divisible group, 
is $\til{s}(\til{G})=s(\til{G}[p])$?
\end{enumerate}
\end{question}

The two parts of Question \ref{Qask} have the same answer by Proposition 
\ref{Pssrank=ssE}.
Here is one difficulty in answering this question.

\begin{remark} 
The $p$-divisible group $\til I_{1,1}$ is isomorphic (over $k$) to the $p$-divisible group $H_{1,1}$ introduced in \cite[5.2]{dejongoort}.  
Consequently, it is {\em minimal} in the sense of \cite[page 1023]{oortminimal}; 
if $\til M$ is any Dieudonn\'e module such that $(\til M \otimes_W k) \simeq M_{1,1}^{\oplus s}$, 
then there is an isomorphism $\til M \simeq \til M_{1,1}^{\oplus s}$.  

In spite of this, because of difficulties with extensions (see, e.g., \cite[Remark 3.2]{oortminimal}), 
one cannot immediately conclude that $\til{M}$ admits $\til M_{1,1}^{\oplus s}$ 
as a summand if $\til M/p \til{M}$ has superspecial rank $s$.  Indeed, Lemma \ref{lemsshrs0} indicates that an appeal to minimality alone is insufficient; any argument must make use of the principal quasipolarization.

\end{remark}

\section{Superspecial rank of supersingular Jacobians} \label{Sjacobian}

If $X/k$ is a (smooth, projective, connected) curve, its superspecial and elliptic ranks are those of its Jacobian:
$s(X)=s({\rm Jac}(X))$ and $e(X)=e({\rm Jac}(X))$.
In this section, we address the question of which superspecial ranks occur for Jacobians of (supersingular) curves.
First, recall that there is a severe restriction on the genus of a superspecial curve.  

\begin{theorem} \label{Tekedahl} (Ekedahl)
If $X/k$ is a superspecial curve of genus $g$, then $g \leq p(p-1)/2$ \cite[Theorem 1.1]{Ekedahl}, see also \cite{Baker}. 
\end{theorem}

For example, if $p=2$, then the genus of a superspecial curve is at most $1$.
The Hermitian curve $X_p: y^p+y=x^{p+1}$ is a superspecial curve realizing the upper bound of Theorem \ref{Tekedahl}.

In Section \ref{Shyp}, we determine the superspecial ranks of all hyperelliptic curves in characteristic $2$.
We determine the superspecial rank of the Jacobians of Hermitian curves in Section \ref{Sherm}. In both cases, this gives an upper bound for the elliptic rank.

\subsection{Supersingular Jacobians} \label{SssJ}

Recall that a curve $X/\FF_q$ is {\it supersingular} if the Newton polygon of $L(X/\FF_q, t)$ is a line segment of slope $1/2$ or, equivalently, if 
the Jacobian of $X$ is supersingular.  One thing to note is that a curve $X/\FF_q$ is supersingular if and only if 
$X$ is minimal over $\FF_{q^c}$ for some $c \in \NN$.

Van der Geer and Van der Vlugt proved that there exists a supersingular curve of every genus in characteristic $p=2$ \cite{VdGVdV}.
For $p \geq 3$, it is unknown if there exists a supersingular curve of every genus.
An affirmative answer would follow from a conjecture about deformations of reducible supersingular curves 
\cite[Conjecture 8.5.7]{oort:rend}.
There are many constructions of supersingular curves having arbitrarily large genus.

Recall (from proof of Proposition \ref{Pexists} and remarks after Proposition \ref{Psselliptic}) 
that there exists a (non-simple) 
supersingular principally polarized abelian variety of dimension $g$ over $k$ 
with elliptic rank $e$ if and only if $0 \leq e \leq g-2$ or $e=g$.
In light of this, one can ask the following question.

\begin{question} \label{QssJac}
Given $p$ prime and $g \geq 2$ and $0 \leq s \leq g-2$,
does there exist a smooth curve $X$ over $\overline{\FF}_p$ of genus $g$ whose Jacobian is supersingular and has elliptic rank $e$?
\end{question}

The answer to Question \ref{QssJac} is yes when $g=2,3$ and $e=0$.
To see this, recall from the proof of Lemma \ref{sswiths=0} that a generic supersingular principally polarized abelian variety 
of dimension $g$ has Dieudonn\'e module $\EE/\EE(F^g+V^g)$, which has superspecial rank $s=0$.
When $g=2,3$, such an abelian variety is the Jacobian of a smooth curve with $e=0$.  

One expects the answer to Question \ref{QssJac} is yes when $g=3$ and $e=1$ also. 
To see this, let $E$ be a supersingular elliptic curve.
Let $A$ be a supersingular, non-superspecial abelian surface.
The $3$-dimensional abelian variety $B= A \times E$ is supersingular and has superspecial rank $1$.
If there is a principal polarization on $B$ which is not the product polarization, then 
$B$ is the Jacobian of a smooth curve.

Question \ref{QssJac} is open for $g\geq 4$.

\subsection{Superspecial rank of hyperelliptic curves when $p=2$} \label{Shyp}

In this section, suppose $k$ is an algebraically closed field of characteristic $p=2$.
Application \ref{App1} states that the superspecial rank of a hyperelliptic curve over $k$ with $2$-rank $0$ is either 0 or 1. 
More generally, Application \ref{App1general} states that the superspecial rank of a hyperelliptic curve over $k$ with $2$-rank $r$ 
is bounded by $1+r$.

A hyperelliptic curve $Y$ over $k$
is defined by an Artin-Schreier equation 
\[y^2+y=h(x),\] for some non-constant rational function $h(x) \in k(x)$.
In \cite{EP13}, the authors determine the structure of the Dieudonn\'e module $M$ of 
${\rm Jac}(Y)$ for all hyperelliptic curves $Y$ in characteristic $2$.
A surprising feature is that the isomorphism class of $M$ depends only on the orders of the poles of $h(x)$, 
and not on the location of the poles or otherwise on the coefficients of $h(x)$.

In particular, consider the case that the $2$-rank of $Y$ is $0$, or equivalently, that $h(x)$ has only one pole.
In this case, the Ekedahl-Oort type is $[0,1,1,2,2, \ldots, \lfloor \frac{g}{2} \rfloor]$ \cite[Corollary 5.3]{EP13}.
The $a$-number is $\lceil \frac{g}{2} \rceil$.

\begin{application} \label{App1}
Let $Y$ be a hyperelliptic curve of genus $g$ with $2$-rank $0$ defined over an algebraically closed field of characteristic $2$.
Then the superspecial rank of ${\rm Jac}(Y)$ is $s=1$ if $g \equiv 1 \pmod 3$ and is $s=0$ otherwise.
The elliptic rank of ${\rm Jac}(Y)$ is $e \leq 1$ if $g \equiv 1 \pmod 3$ and $e=0$ otherwise.
\end{application}

\begin{proof}
This follows by applying the algorithm in \cite[Section 5.2]{EP13}.  
Specifically, by \cite[Proposition 5.10]{EP13} (where $c=g$),
the Dieudonn\'e module of the group scheme with Ekedahl-Oort type
$[0,1,1,2,2, \ldots, \lfloor \frac{g}{2} \rfloor]$ is generated by variables $X_j$ for 
$\lceil (g+1)/2 \rceil \leq j \leq g$ subject to the relations $F^{e(j)+1}(X_j)+V^{\epsilon(\iota(j))+1}(X_{\iota(j)})$, 
where the notation is defined in \cite[Notation 5.9]{EP13}.  Then $M_{1,1}$ occurs as a summand 
if and only if there is some $j$ such that $e(j)=\epsilon(j)=0$ and $j=\iota(j)$.  
The condition $e(j)=0$ is equivalent to $j$ being odd.  
The conditions $\epsilon(j)=0$ and $j=\iota(j)$ imply that $2g-2j+1=g-(j-1)/2$ which is possible only if 
$g \equiv 1 \bmod 3$.  If $g \equiv 1 \bmod 3$, then $j=(2g+1)/3$ so the maximal rank of a summand 
isomorphic to $I_{1,1}^s$ is $s=1$.
\end{proof}

\begin{remark}
It is not known exactly which natural numbers $g$ can occur as the genus of a supersingular hyperelliptic curve 
over $\overline{\FF}_2$. 
On one hand, if $g=2^s-1$, then there does not exist a supersingular hyperelliptic curve of genus $g$ over 
$\overline{\FF}_2$ \cite{SZss}.

On the other hand, if $h(x)=xR(x)$ for an additive polynomial $R(x)$ of degree $2^{s}$, then 
$Y$ is supersingular of genus $2^{s-1}$ \cite{VdGVdV92}.
If $s$ is even, then Application \ref{App1} shows that ${\rm Jac}(Y)$ has no elliptic curve factors in 
a decomposition up to isomorphism, 
even though it decomposes completely into elliptic curves up to isogeny. 
\end{remark}

More generally, we now determine the superspecial ranks of hyperelliptic curves in characteristic $2$ having arbitrary $2$-rank.
Consider the divisor of poles 
\[{\rm div}_\infty (h(x)) = \sum_{j=0}^{r} d_j P_j.\] 
By Artin-Schreier theory, one can suppose that $d_j$ is odd for all $j$. 
Then ${\rm Jac}(Y)$ has genus $g$ satisfying $2g+2=\sum_{j=0}^r (d_j+1)$ by the Riemann-Hurwitz formula \cite[IV, Prop.\ 4]{Se:lf}
and has $2$-rank $f=r$ by the Deuring-Shafarevich formula \cite[Theorem 4.2]{Subrao} or \cite[Cor.\ 1.8]{Crew}.
These formulae imply that,
for a given genus $g$ (and $2$-rank $r$), there is another discrete invariant of a hyperelliptic curve $Y/k$, 
namely a partition of $2g+2$ into $r+1$ positive even integers $d_j + 1$.
In \cite{EP13}, the authors prove that the Ekedahl-Oort type of $Y$ depends only on this discrete invariant.

Specifically, consider the variable $x_j:=(x-P_j)^{-1}$, which is the inverse of a uniformizer at the branch point $P_j$ in ${\mathbb P}^1$ 
(with $x_j=x$ if $P_j = \infty$).
Then $h(x)$ has a partial fraction decomposition of the form
\[h(x)=\sum_{j=0}^r h_{j} \big(x_j\big),\]
where $h_j(x) \in k[x]$ is a polynomial of degree $d_j$.
Let $c_j=(d_j-1)/2$ and note that $g=r+\sum_{j=0}^r c_j$.
For $0 \leq j \leq r$, consider the Artin-Schreier $k$-curve $Y_j$ with affine equation $y^2 - y = h_j(x)$.
Let $E_0$ be an ordinary elliptic curve over $k$.

Then \cite[Theorem 1.2]{EP13} states that the de Rham cohomology of $Y$ decomposes, as a module under the actions of Frobenius $F$ and Verschiebung $V$, 
as: 
\[
H^1_{\rm dR}(Y) \simeq  H^1_{\rm dR}(E_0)^{r} \oplus \bigoplus_{j=0}^r H^1_{\rm dR}(Y_j).
\]

Since $E_0$ is ordinary, it has superspecial rank $0$.  
The superspecial rank of ${\rm Jac}(Y)$ is thus the sum of the superspecial ranks of ${\rm Jac}(Y_j)$.
Applying Application \ref{App1} to $\{Y_j\}_{j=0}^r$ proves the following.

\begin{application} \label{App1general} 
Consider a hyperelliptic curve $Y$ defined over an algebraically closed field of characteristic $2$.
Then $Y$ is defined by an equation of the form
$y^2+y=h(x)$ with ${\rm div}_\infty (h(x)) = \sum_{j=0}^{r} d_j P_j$ and $d_j$ odd.  
Recall that $Y$ has genus $g=r+\sum_{j=0}^r c_j$ where $c_j=(d_j-1)/2$ and $p$-rank $r$.
The superspecial rank of ${\rm Jac}(Y)$ equals the number of $j$ such that $c_j \equiv 1 \bmod 3$.
In particular, $s({\rm Jac}(Y)) \leq 1 + r$ and $e(({\rm Jac}(Y))) \leq 1 + 2r$.
\end{application}

\subsection{Hermitian curves} \label{Sherm}

The last examples of the paper are about the superspecial rank for one of the three classes of (supersingular) Deligne-Lusztig curves: the Hermitian curves $X_q$ for $q=p^n$ for an arbitrary prime $p$.
In most cases, the superspecial (and elliptic) ranks are quite small, which is somewhat surprising since these curves are exceptional from many perspectives.

Let $q=p^n$.  The Hermitian curve $X_q$ has affine equation 
\[y^q + y = x^{q+1}.\]
It is supersingular with genus $g=q(q-1)/2$.
It is maximal over $\FF_{q^2}$ because $\#X_q\left(\FF_{q^2}\right)=q^3+1$.
The zeta function of $X_q$ is \[Z(X_q/\FF_q, t)=\frac{(1+qt^2)^g}{(1-t)(1-qt)}.\]


In fact, $X_q$ is the unique curve of this genus which is maximal over $\FF_{q^2}$ \cite{ruckstich}.
This was used to prove that $X_q$ is the Deligne-Lusztig variety for ${\rm Aut}(X_q)={\rm PGU}(3,q)$  \cite[Proposition 3.2]{HansenDL}.


By \cite[Proposition 14.10]{Gross}, the $a$-number of $X_q$ is 
\[a=p^n(p^{n-1}+1)(p-1)/4,\] 
which equals $g$ when $n=1$, equals $g/2$ when $n=2$, and is approximately $g/2$ for $n \geq 3$.
In particular, $X_{p^n}$ is superspecial if and only if $n=1$.

In \cite{PW12}, for all $q=p^n$, the authors determine the Dieudonn\'e module $\dieu_*(X_q) = \dieu_*(\operatorname{Jac}(X_q)[p])$, complementing earlier work in \cite{dum95, dum99}.
In particular, \cite[Theorem 5.13]{PW12} states that the distinct indecomposable factors of Dieudonn\'e module $\dieu_*(X_q)$
are in bijection with orbits of $\ZZ/(2^n+1) -\{0\}$ under $\times 2$.
Each factor's structure is determined by the combinatorics of the orbit, which depends only on $n$ and not on $p$. 
The multiplicities of the factors do depend on $p$.
For example, when $n=2$, the Dieudonn\'e module of $X_{p^2}$ is $M_{2,2}^{g/2}$, which has superspecial rank $0$ (Lemma \ref{lemsshrs0}).
Here is an application of these results.

\begin{application} \label{App2}
The elliptic rank of the Jacobian of the Hermitian curve $X_{p^n}$ equals $0$ if $n$ is even and is at most $(\frac{p(p-1)}{2})^n$ if $n$ is odd.
\end{application}
  
\begin{proof}
By Proposition \ref{Psselliptic}, $e({\rm Jac}(X_{p^n})) \leq s({\rm Jac}(X_{p^n}))$.
Applying \cite[Application 6.1]{PW12}, the factor $\EE/\EE(F+V)$ occurs in the Dieudonn\'e module if and only if 
there is an orbit of length 2 in $\ZZ/(2^n+1)$ under $\times 2$.  
This happens if and only if there is an element of order three in $\ZZ/(2^n+1)$, which is true
if and only if $n$ is odd.  If $n$ is odd, this shows that $\EE/\EE(F+V)$ is not a factor of the Dieudonn\'e module and 
$s({\rm Jac}(X_{p^n}))=0$.
If $n$ is even, the multiplicity of this factor is $s({\rm Jac}(X_{p^n}))=(\frac{p(p-1)}{2})^n$.
\end{proof}

\end{document}